\documentclass
[
    a4paper,
    DIV=11,
    abstracton,
]
{scrartcl}

\usepackage
{
    graphicx,
    amssymb,
    amsmath,
    amsthm,
    dsfont, %
    xcolor,
    mathtools,
    kbordermatrix,
    authblk,
    enumitem,
    tikz,
    todonotes,
    algpseudocode,
    booktabs,
    stmaryrd, %
    xspace, %
}

\usepackage[pdffitwindow=false,
            plainpages=false,
            pdfpagelabels=true,
            pdfpagemode=UseOutlines,
            pdfpagelayout=SinglePage,
            bookmarks=false,
            colorlinks=true,
            hyperfootnotes=false,
            linkcolor=blue,
            urlcolor=blue!30!black,
            citecolor=green!50!black]{hyperref}

\usepackage[bf,normal]{caption}
\usepackage{subcaption}
\usepackage{ulem} %
\usepackage[ruled]{algorithm2e} %

\newcommand{\hiprod}[2]{\langle {#1, #2} \rangle _{\mathcal H}}
\newcommand{\iprod}[2]{\langle {#1, #2} \rangle _{\mathcal H}}
\newcommand{\hnorm}[1]{\|{#1}\|_{\mathcal H}}

\newcommand{\rkhs}{\mathbb{H}}
\newcommand{\mmd}{\operatorname{MMD}}

\newcommand{\Kxx}{{K_{XX}}}

\newcommand{\Kxtilde}{{K_{\tilde{X} \tilde{X}}}}
\newcommand{\Kxhat}{{K_{\hat{X} \hat{X}}}}

\newcommand{\avg}[1]{\ensuremath{\frac1{#1}\sum_{i=1}^{#1}}}

\newcommand{\E}{\operatorname{\mathbb E}}

\newcommand{\PhatTauHlmbd}{\hat{\mathcal{P}}^t_{\mathbb{H}, \lambda}}
\newcommand{\PtildeTauHlmbd}{\tilde{\mathcal{P}}^t_{\mathbb{H}, \lambda}}
\newcommand{\PTauHlmbd}{{\mathcal{P}}^t_{\mathbb{H}, \lambda}}

\newcommand{\EpNorm}[1]{\|{\mathcal{E}}p_{#1}\|_\rkhs}

\newcommand{\LH}{{L(\mathbb{H})}}

\newcommand{\R}{\mathbb{R}}                                     %
\newcommand{\innerprod}[2]{\left\langle #1,\, #2 \right\rangle} %
\newcommand{\ts}{\hspace*{0.1em}}                               %
\providecommand{\norm}[1]{\left\lVert #1 \right\rVert}          %
\DeclareMathOperator{\spn}{span}

\newcommand{\HS}{\mathrm{HS}}									%

\newtheorem{theorem}{Theorem}[section]
\newtheorem{problem}{Problem}[section]
\newtheorem{corollary}[theorem]{Corollary}
\newtheorem{lemma}[theorem]{Lemma}
\newtheorem{proposition}[theorem]{Proposition}
\newtheorem{definition}[theorem]{Definition}
\theoremstyle{definition}
\newtheorem{example}[theorem]{Example}
\newtheorem{remark}[theorem]{Remark}
\newtheorem{assumption}{Assumption}

\mathtoolsset{centercolon} %

\allowdisplaybreaks

\title{Propagating Kernel Ambiguity Sets in Nonlinear Data-driven Dynamics Models}
\author{Jia-Jie Zhu}
\affil{Weierstrass Institute for Applied Analysis and Stochastics, Mohrenstraße 39, 10117\\
Berlin, Germany}

\date{}

\begin{document}
\maketitle
\begin{abstract}
This paper provides answers to an open problem:
given a nonlinear data-driven dynamical system model, e.g., kernel conditional mean embedding (CME) and Koopman operator,
how can one propagate the ambiguity sets forward for multiple steps?
This problem is the key to solving distributionally robust control and learning-based control of such learned system models under a data-distribution shift.
Different from previous works that use either static ambiguity sets, e.g., fixed Wasserstein balls,
or dynamic ambiguity sets under known piece-wise linear (or affine) dynamics,
we propose an algorithm that exactly propagates ambiguity sets through nonlinear data-driven models using the Koopman operator and CME, via the kernel maximum mean discrepancy geometry.
Through both theoretical and numerical analysis, we show that our kernel ambiguity sets are the natural geometric structure for the learned data-driven dynamical system models.

\end{abstract}

\section{Introduction}

The goal of this paper is to quantify ambiguity\textemdash the uncertainty of uncertainty description\textemdash in data-driven dynamics models for stochastic differential equations (SDE), e.g., data-driven approximation of Koopman operator~\cite{Mezic05} and kernel condition mean embedding~\cite{song2009hilbert}.
Intuitively, we consider the problem
\begin{problem}
    [One-step ambiguity set propgation]
    \label{problem-1}
Given ambiguity set $\mathcal M _t$ at time $t$ and the nonlinear data-driven dynamics model
$X_{t+1}\in \hat{f}(X_{t})$, find the ambiguity set $\mathcal M _{t+1}$ at the next time step.
\end{problem}

One important impact area of ambiguity quantification is distributionally robust risk-averse decision-making.
There, to hedge against distributional ambiguity,
\emph{distributionally robust optimization} (DRO) solves a minimax robustifed stochastic program 
\cite{delageDistributionallyRobustOptimization2010a}
\begin{align}
    \min_\theta \sup_{\mu\in \mathcal M} \E_{x\sim\mu} l(\theta; x),
\end{align}
whose central ingredient \emph{ambiguity set} $\mathcal M$ describes the uncertainty of underlying probability distribution.
In recent cutting-edge studies in DRO, researchers have invented reformulation techniques for solving DRO with geometries in the probability spaces,
e.g., 
Wasserstein metric \cite{mohajerin2018data}, i.e.,
with ambiguity balls of distributions
$
    \left\{
        \mu\ |\ W_p(\mu, \hat \mu)\leq \rho
    \right\}
$
centered at an empirical data distribution $\hat \mu$,
where $W_p$ is the $p$-Wasserstein distance~\cite{santambrogio2015optimal}.
Similar results also exist for $f$-divergences~\cite{ben-talRobustSolutionsOptimization2013},
and, most relevant to this paper, the kernel maximum mean discrepancy (MMD)~\cite{zhu2021kernel}.
Compared with ERM,
those ambiguity sets describe the level of trust in the empirical data samples in the corresponding geometry.
In the context of systems and control,
recent works in \emph{distributionally robust control}
exploit aforementioned DRO reformulation techniques \cite{mohajerin2018data}
in optimal control problems, e.g., \cite{yangWassersteinDistributionallyRobust2021a}.
In the context of stochastic dynamics, for example, ambiguity sets under known dynamics have been studied using information divergences~\cite{hansen2022structured} and the Wasserstein distance~\cite{shafieezadehabadehWassersteinDistributionallyRobust2018}.

\begin{figure}[t!]
    \centering
            \includegraphics[width=8cm]{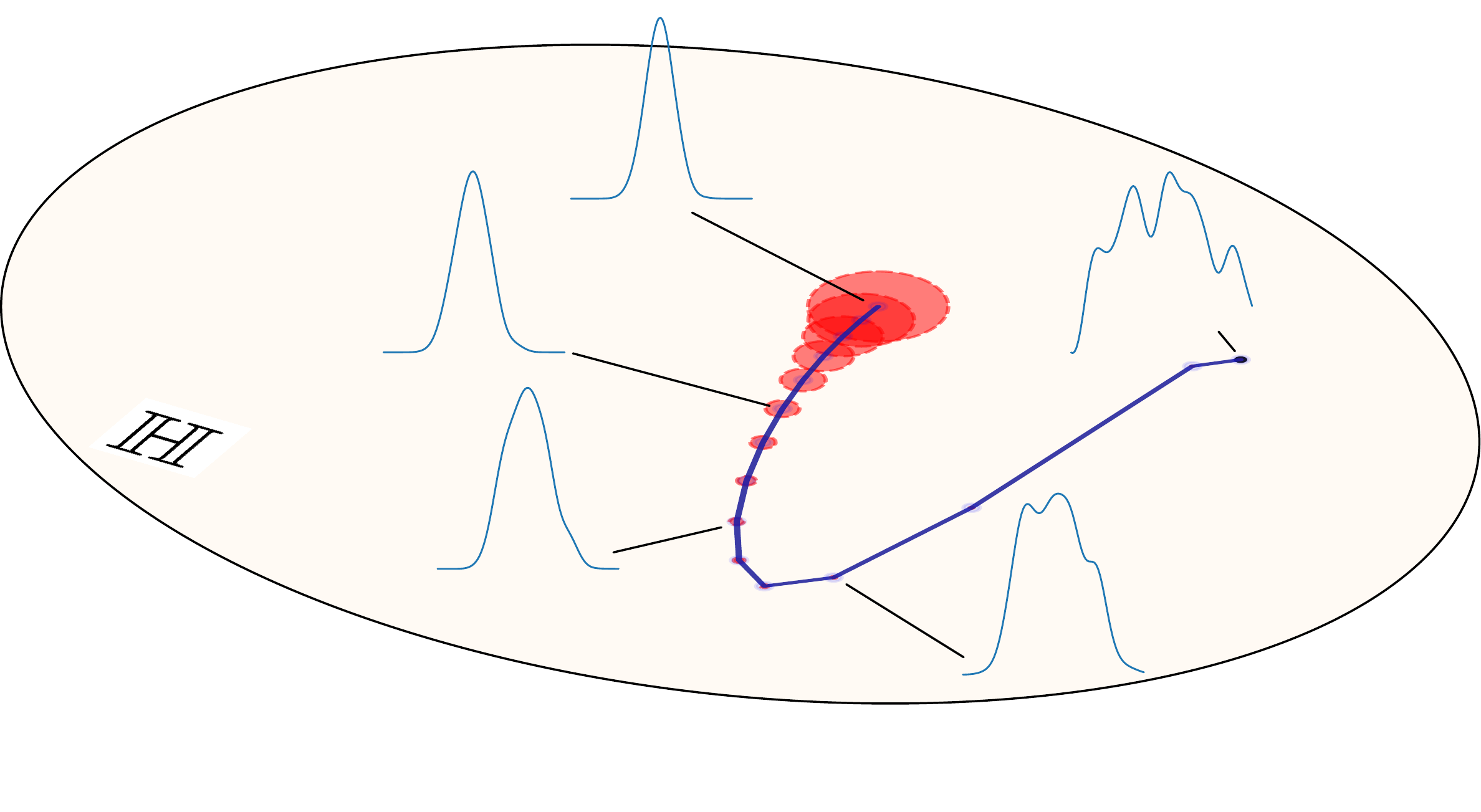}
    \caption{
        Embedded evolution of the system state distribution in a reproducing kernel Hilbert space (RKHS) $\mathbb H$.
        The blue curve is the evolution path of the embedded state distribution in $\mathbb H$. Each node of the path denotes an embedded system state distribution.
        Furthermore, this paper proposes an approach to bound the multistep deviation from the true (unknown) generating process.
        This deviation bound is illustrated by the red ambiguity balls centered at the propagated state distributions along the path.
        The radius is the deviation bound in the Hilbert norm.
    }
    \vspace{-0.7cm}
    \label{fig-embedded-evo-intro}
\end{figure}
However, despite the already sizable body of literature, the current state-of-the-art learning-based control falls short of distributionally robust control of 
uncertain data-driven dynamics models, which are mostly nonlinear such as the 
Gaussian process models \cite{hewingLearningBasedModelPredictive2020a},
and data-driven approximation of the Koopman operator \cite{Nueske2021,KM18a}.
While those methods are uncertainty-aware,
they do not consider that the probability distribution themselves might be subject to a second layer of uncertainty, the distributional ambiguity.
This is due to the technical difficulty that
\emph{there exists no method to propagate ambiguity sets,
such as Wasserstein balls,
through learned nonlinear data-driven dynamics models
}.
The following motivating example highlights the technical difficulty and open questions.
\begin{example}
    \label{ex-krr}
    Suppose the dynamics is learned using non-parametric kernel regression, i.e., 
    \begin{align}
        \label{eq-krr}
        \hat{f}(x) =   \hat{Y} (k(\hat{X}, \hat{X})+m\lambda I)^{-1} k (\hat{X}, x)
        ,
    \end{align}
    which is also equivalent to the Gaussian process mean dynamics~\cite{kanagawaGaussianProcessesKernel2018a}.
    $k$ is assumed to be a positive definite kernel, e.g., radial basis kernel.
    Given a Wasserstein metric ball ambiguity set at time $t$,
    $$
    X_t\sim \mu_t, \mu_t\in \mathcal M_t:=    
        \left\{
            \mu\ |\ W_2(\mu_t, \hat \mu_t)\leq \rho_t
        \right\} 
    $$
    Then, there exists no tractable method to compute the next time step Wasserstein ambiguity set $\mathcal M_{t+1}$ for the state distribution
    $X_{t+1}\sim \mu_{t+1}, \mu_{t+1}\in \mathcal M_{t+1}$
    .
\end{example}

To fill this gap,
this paper combines the \emph{ambiguity sets} from distributionally robust optimization
, with the \emph{nonlinear data-driven modeling} techniques from Koopman operators and kernel conditional embedding.
We summarize our {contributions} below and provide a brief overview of the results.
    \paragraph*{1}
    First and foremost, we propose an algorithm to propagate
    a distributional ambiguity set through data-driven Koopman dynamics models and, equivalently, conditional mean embedding models.
    Concretely, 
    given learned embedded data-driven Koopman operator model
    in the form of its adjoint PF operator $\PhatTauHlmbd$,
    we solve a multistep version of Problem~\ref{problem-1}
    under the kernel maximum mean discrepancy metric via Algorithm~\ref{alg:mmd-prop-forward}.
    Different from the one-step error analysis in the literature
    ~\cite{parkMeasuretheoreticApproachKernel2020,
    Li2022,
    kosticLearningDynamicalSystems2022},
    we construct the multistep distributional \emph{ambiguity tube} of entire trajectories\textemdash a set of distributions centered around the empirical path of the embedded stochastic system
        $
        \mathcal M_t
        :=    
        \left\{
                \mu\ |\ \mmd(\mu, \hat{\mathcal{E}}\hat{p}_t )\leq \rho_t
            \right\}
            ,
            \quad
            t=0,1, \dots, T,
        $
        where $\left\{\hat{\mathcal{E}}\hat{p}_t\right\}_{t=0}^{T}$ is the path of the data-driven model obtained using the model $\PhatTauHlmbd$, and $\rho_t$ are the error bounds quantified in this paper.
        To the best of our knowledge, this is the first result of this type.
        To help understand this main contribution of the paper,
        we illustrate the intuition of the embedded evolution, along with its multistep error, in an RKHS in Figure~\ref{fig-embedded-evo-intro}. See the caption for more detail.
    
        \paragraph*{2}
    Unlike the Wasserstein ambiguity sets whose radius is hard to set in practice, unless using costly procedures such as cross-validation~\cite{mohajerin2018data}, we propose a computationally efficient bootstrap procedure to estimate the ambiguity tube suitable for practical nonlinear dynamics simulation; see Algorithm~\ref{alg:bootstrap}.
    To the best of our knowledge, this is the first computable multistep error estimator in the literature for nonlinear data-driven dynamical systems.
    Via numerical experiments, we demonstrate that our multistep bootstrap estimator 
    in the propagation of MMD ambiguity sets for stochastic systems.

\section{Preliminaries}
\label{sec:prelim}
\subsection{Data-driven Modeling using Embedded Koopman and Perron-Frobenius Operators}
\label{sec:Koopman operator}

A data-driven dynamics model closely related to \eqref{eq-krr}  
is
the data-driven approximation of \emph{Koopman operators} or \emph{Perron-Frobenius operators (PF operators)} for stochastic differential equations using reproducing kernel Hilbert spaces. For a general dynamical system, the Koopman operator is the linear conditional expectation operator mapping a given observable function to its expectation with respect to the conditional distribution after evolving the dynamical system over a time window $t \geq 0$~\cite{Ko31,Mezic05}, and the PF operator is its adjoint. Due to their straightforward approximability based on simulation data, these evolution operators have emerged as a powerful tool for data-driven modeling, analysis, model reduction and control of complex dynamical systems, see~\cite{BMM12,KNKWKSN18,Mauroy2020} for recent reviews on these topics. Major applications include, among many others, fluid dynamics and molecular dynamics simulations.

Due to the close connection of the Koopman approach to machine learning, a natural extension is to represent evolution operators on a reproducing kernel Hilbert space (RKHS)~\cite{Aronszajn1950,Schoelkopf2002}, which can serve as an infinite-dimensional and therefore powerful approximation space. 
This was first suggested in~\cite{WRK15}, a rigorous connection to the concept of \emph{kernel conditional mean embedding} (CME)~\cite{Fukumizu2004} was made in~\cite{Klus2020_RKHS}. In particular, it was shown that one can realize a view of the exact dynamics through the lens of an \emph{embedding operator} by a linear operator acting only on the RKHS, which can be estimated from data using kernel evaluations.
Intuitively, CME can be viewed as the infinite-dimensional vectorial version of the kernel regression~\eqref{eq-krr}, making it suitable for modeling SDE and PDE systems.

We consider SDE dynamics with governing equation
\begin{equation}
    \label{eq:SDE}
    \mathrm{d}X_t = b(t, X_t) \ts \mathrm{d}t + \sigma(t, X_t)\ts \mathrm{d}W_t,
\end{equation}
where $ W_t $ is the $ d $-dimensional Wiener process.
Let $ \mathbb{X} $ be the state space, which is a subset of $d$-dimensional Euclidean space, $ \mathbb{X} \subset \R^d $.
For example, the Langevin SDE is given by
\begin{equation*} 
    \mathrm{d}X_t = -{\nabla}V(X_t) \ts \mathrm{d}t + \sqrt{2\beta^{-1}} \ts \mathrm{d}W_t,
\end{equation*}
where, $ V \colon \mathbb{X} \to \R $ is the potential energy, $\beta > 0$ is the diffusion constant (often related to temperature in statistical physics).
In this paper, we are particularly interested in the Langevin SDE due to the characterization of the stability property below.

\begin{assumption}
\label{a:existence_sde}
We assume $V$ is smooth enough such that global existence and uniqueness of solutions to~\eqref{eq:SDE} is guaranteed, see~\cite{Oksendal2003}. Moreover, we also assume that, if $\mathbb{X}$ is unbounded, then $V$ grows sufficiently fast at infinity to ensure that $\mathrm{d}\mu(x) \sim \exp(-\beta V(x)) \,\mathrm{d}x$ is finite and is the unique invariant measure for $X_t$.
\end{assumption}

We use $L^2_\mu(\mathbb{X})$ for the weighted $L^2$-space associated with the measure $\mu$, and $H^1_\mu(\mathbb{X})$ for the Sobolev space of functions with first order weak derivatives in $L^2_\mu(\mathbb{X})$.

\begin{assumption}
\label{a:poincare_ineq}
The potential $V$ satisfies a \emph{Poincaré inequality} with constant $R > 0$, that is, for all $\phi \in H^1_\mu(\mathbb{X})$ such that $\int_\mathbb{X} \phi(x)\,\mathrm{d}\mu(x) = 0$, we have
\begin{equation*}
\| \phi \|_{L^2_\mu(\mathbb{X})} \leq \frac{1}{2R}\|\nabla \phi \|_{L^2_\mu(\mathbb{X})}
\end{equation*}
\end{assumption}
Intuitively, Poincaré inequality ensures that the driving energy of the underlying dynamical system has favorable geometric properties.
It is often used to characterize the convergence rate of the system state distribution of \eqref{eq:SDE} in, e.g., the $\chi^2$-divergence.
Conditions on the potential to ensure a Poincaré inequality have been well-studied in the literature, see the book~\cite{BAKRY2013} and the review~\cite{Lelievre2016} for exhaustive discussions.

\paragraph{Evolution Operators} Associated to the invariant measure is the conditional expectation operator $\mathcal{K}^t: L^2_{\mu}(\mathbb{X}) \mapsto L^2_{\mu}(\mathbb{X})$, given by
\[ \mathcal{K}^t \phi(x) = \mathbb{E}^x[ \phi(X_t)], \quad \phi \in L^2_{\mu}(\mathbb{X}), \]
where the expectation is taken over the state $X_t$ conditioned on $X_0 = x$.
The conditional expectation operator is also known as the \emph{Koopman operator} with respect to the invariant measure. The study of Koopman operators has received significant attention in recent years due to its close connection to data-driven methods. A consequence of Assumption~\ref{a:poincare_ineq} is
\begin{proposition}[\cite{Lelievre2016}]
The Koopman operator $\mathcal{K}^t$ is \emph{exponentially stable}, that is, for all $\phi \in L^2_\mu(\mathbb{X})$ such that $\int_\mathbb{X} \phi(x)\,\mathrm{d}\mu(x) = 0$:
\[  \|\mathcal{K}^t \phi \|_{L^2_\mu(\mathbb{X})} \leq e^{-\frac{2R t}{\beta}} \|\phi\|_{L^2_\mu(\mathbb{X)}}. \]
\end{proposition}

\subsection{Reproducing Kernel Hilbert Spaces}
\label{subsec:kernel_edmd}
In this paper, we consider representations of evolution operators on a reproducing kernel Hilbert space (RKHS), see e.g.~\cite{Aronszajn1950,Schoelkopf2002}. Let $k$ be a symmetric and positive-definite kernel, and $\mathbb{H}$ be the associated RKHS. The corresponding feature map is denoted $\Phi: \mathbb{X} \mapsto \mathbb{H}, \, x \mapsto \Phi(x) := k(x, \cdot)$. We now state the main assumptions on the RKHS required in this paper \cite{Steinwart2012}:
\begin{assumption}
\label{a:basics_rkhs}
(i) The norm of the feature map is in $L^2_\mu(\mathbb{X})$:
\[ \kappa^2 := \int_\mathbb{X} \|\Phi(x)\|^2_\mathbb{H}\,\mathrm{d}\mu(x) = \int_\mathbb{X} k(x, x) \,\mathrm{d}\mu(x) < \infty. \]
(ii) The inclusion map $\iota: \,\mathbb{H} \mapsto L^2_\mu(\mathbb{X})$ is injective.
\end{assumption}

Part (i) of Assumption~\ref{a:basics_rkhs} implies that the RKHS is compactly embedded into $L^2_\mu(\mathbb{X)}$, i.e. $\iota$ is compact. Moreover, its adjoint is given by the \emph{integral operator} associated with the RKHS $\mathbb{H}$,
\begin{align*}
\mathcal{E}(\phi) = \iota^* \phi = \int_\mathbb{X} \phi(x) \Phi(x) \,\mathrm{d}\mu(x).
\end{align*}
Both $\iota$ and $\iota^*$ are Hilbert-Schmidt, hence the concatenation $G := \iota \iota^*$ is trace class. By the spectral theorem, there is a complete orthonormal system of eigenfunctions $\phi$ of $G$, associated with strictly positive eigenvalues (by Part (ii)).

\subsection{Embedded evolution operators and conditional mean embedding}
\label{subsec:embedded_ops}
We now recall the learning framework for evolution operators using their embeddings into reproducing kernel Hilbert spaces, often referred to as the kernel conditional mean embedding, see~\cite{Fukumizu2004,Fukumizu2013,Klus2020_RKHS}. The starting point is the following rank-one operators:
\begin{definition}[rank-one operators]:
For $x, y \in \mathbb{X}$, define the following rank-one operators on $\mathbb{H}$:
\begin{align*}
c_{xx} &= \innerprod{\Phi(x)}{\cdot}_\mathbb{H}\Phi(x), & c_{xy} &= \innerprod{\Phi(y)}{\cdot}_\mathbb{H}\Phi(x).
\end{align*}
\end{definition}

\begin{proposition}[\cite{Klus2020_RKHS}]
\label{prop:existence_covariance_ops}
Under Assumption~\ref{a:basics_rkhs}, the following co-variance and cross-covariance operators are Hilbert-Schmidt operators on $\mathbb{H}$:
\begin{align*}
\mathcal{C}_{XX} &:= \int_\mathbb{X} c_{xx} \,\mathrm{d}\mu(x), &
\mathcal{C}_{XY} &:= \int_\mathbb{X}\int_\mathbb{X} c_{xy} \,\mathrm{d}\mu_{0,t}(x, y),
\end{align*}
where $\mu$ is the probability measure of the initial state $X_0$, and
$\mu_{0,t}$ is the joint probability measure of $X_0$ and the $X_t$.

Moreover, we have for all $\phi,\,\psi \in \mathbb{H}$:
\begin{align}
\label{eq:comm_cov_operators}
\innerprod{ \phi}{\mathcal{C}_{XX}\psi}_\mathbb{H} &= \innerprod{\phi}{\psi}_{\mu}, &
 \innerprod{\phi}{\mathcal{C}_{XY} \psi}_\mathbb{H} &= \innerprod{\phi}{\mathcal{K}^t \psi}_{\mu}.
\end{align}
\end{proposition}
This embedding means that the evolution described by the Koopman operator can now be studied within the embedding RKHS $\mathbb{H}$.
\begin{remark}
Strictly speaking, we should write
\begin{align*}
\innerprod{\phi}{\mathcal{C}_{XX}\psi}_\mathbb{H} &= \innerprod{\iota\phi}{\iota\psi}_{\mu}, &
 \innerprod{\phi}{\mathcal{C}_{XY} \psi}_\mathbb{H} &= \innerprod{\iota\phi}{\mathcal{K}^t \iota\psi}_{\mu}
\end{align*}
instead of~\eqref{eq:comm_cov_operators}, using the inclusion map $\iota$. However, we will only make this distinction when a precise distinction between functions in $\mathbb{H}$ and their $L^2_\mu$ representatives is required.
\end{remark}

We now discuss the numerical approximation of embedded evolution operators. First, for any two matrices of $m$ data points $X, Y \in \mathbb{R}^{d\times m}$, we use the notation $K_{XY} \in \mathbb{R}^{m\times m}$ to denote the matrix of all pairwise evaluations of the kernel $k$, i.e.,
$K_{XY} = \left[ k(x_i, y_j)\right]_{i,j=1}^m$.
Also, we denote the linear span of the feature maps for all data points in a collection $X$ by $\mathbb{H}_X$, that is
$\mathbb{H}_X = \spn\{\Phi(x_k)\}_{k=1}^m$.
The features $\Phi(x_k)$ serve as a canonical basis for this space, we denote the formal $m$-dimensional vector of all these functions by $\Phi_X$ (or $\Phi_X(x)$ if evaluated at a point $x$).
Now, let $\{x_k\}_{k=1}^m$ be a collection of data points sampling the measure $\mu$, and let $\{y_k\}_{k=1}^m$ be obtained by integrating the dynamics over time $t$ from $x_k$. Alternatively, $x_k$ and $y_k$ can also be chosen as time-lagged samples from a single ergodic trajectory of the dynamics~\eqref{eq:SDE}. We then define
\begin{definition}
The empirical counterparts of the co-variance operators are denoted by $\hat{\mathcal{E}}: L^2_\mu \mapsto \mathbb{H}$ and $\hat{\mathcal{C}}_{XX}, \, \hat{\mathcal{C}}_{YX} :\, \mathbb{H} \mapsto \mathbb{H}$. Their definitions are:
\begin{align*}
\hat{\mathcal{E}} \phi &= \frac{1}{m}\sum_{k=1}^m \phi(x_k) \Phi(x_k), & \hat{\mathcal{C}}_{XX} \phi &= \frac{1}{m}\sum_{k=1}^m c_{x_k, x_k}, & \hat{\mathcal{C}}_{YX} \phi &= \frac{1}{m}\sum_{k=1}^m c_{y_k, x_k}.
\end{align*}
\end{definition}

In the literature of kernel methods for machine learning, the operation $\mathcal{E}$ and $\hat{\mathcal{E}}$ is typically referred to as \emph{kernel mean embedding}.
Explicit calculation of the inverse operator $\mathcal{C}_{XX}^{-1}$ or $\hat{\mathcal{C}}_{XX}^{-1}$ is often avoided by means of regularization. For $\lambda > 0$, \emph{regularized embedded operators} are defined by
\begin{align*}
\mathcal{P}^t_{\mathbb{H}, \lambda} &= \mathcal{C}_{YX} (\mathcal{C}_{XX} + \lambda \mathrm{Id})^{-1}, &
\hat{\mathcal{P}}^t_{\mathbb{H}, \lambda} &= \hat{\mathcal{C}}_{YX} (\hat{\mathcal{C}}_{XX}+ \lambda \mathrm{Id})^{-1}.
\end{align*}
These operators are automatically well-defined and bounded on all of $\mathbb{H}$, even without the invariance Assumption~\ref{a:existence_sde}. Furthermore, we note that any function $\phi$ orthogonal to $\mathbb{H}_X$ satisfies $\phi(x_k) = 0$ for all $k$, which implies that both $\hat{\mathcal{C}}_{XX}$ and $\hat{\mathcal{C}}_{YX}$ vanish on $\mathbb{H}_X^\perp$. Therefore, we also have $\hat{\mathcal{P}}^t_{\mathbb{H}, \lambda} \equiv 0$ on $\mathbb{H}_X^\perp$, and it is sufficient to consider $\hat{\mathcal{P}}^t_{\mathbb{H}, \lambda}$ as a map between the finite-dimensional spaces $\mathbb{H}_X$ and $\mathbb{H}_Y$.

\paragraph{Matrix Representations}
The matrix representations of these operators are obtained as follows. For a function $\psi = \Phi_X^T \alpha \in \mathbb{H}_X$,
$\alpha\in \mathbb{R}^m$,
we can verify that
\begin{align}
\hat{\mathcal{C}}_{XX} \psi &= \frac{1}{m}\sum_{k=1}^m \Phi(x_k) \left[(K_{XX})_{k, :} \cdot \alpha\right], & \hat{\mathcal{C}}_{YX} \psi &= \frac{1}{m}\sum_{k=1}^m \Phi(y_k) \left[(K_{XX})_{k, :}\cdot \alpha\right]
,
\\
\hat{\mathcal{C}}^0_{XX} \psi &= \frac{1}{m}\sum_{k=1}^m \Phi(x_k) \left[(K_{XX}^0)_{k, :} \cdot \alpha\right], & \hat{\mathcal{C}}^0_{YX} \psi &= \frac{1}{m}\sum_{k=1}^m \Phi(y_k) \left[(K_{XX}^0)_{k, :}\cdot \alpha\right],
\end{align}
where the matrix $K_{XX}^0$ is obtained by subtracting the vector $k_X$ of column sums of $K_{XX}$ from the kernel matrix:
\begin{align*}
K_{XX}^0 &= K_{XX} - \frac{1}{m} \mathds{1} \otimes k_X, & [k_X]_s &= \sum_{r=1}^m k(x_r, x_s).
\end{align*}
Concatenating both representations, the regularized empirical operators possess the matrix representations (with respect to the canonical bases of $\mathbb{H}_X, \, \mathbb{H}_Y$:
\begin{align}
\label{eq:empirical_estimator_hxy}
\hat{P}^t_{XY, \lambda} &= K_{{Y}X} (K_{XX} + m \lambda \mathrm{Id})^{-1}.
\end{align}

\subsection{Kernel maximum mean discrepancy and ambiguity sets}
The kernel \emph{maximum mean discrepancy (MMD)} is defined as the difference between integral maps of the kernel functions
measured in the aforementioned reproducing kernel Hilbert space (RKHS) norm associated with the kernel function $k$.
\begin{equation}
    \mmd(\mu, \nu):=\|\int k(x, \cdot) d\mu -  \int k(x, \cdot) d\nu\|_{\rkhs}.
\end{equation}
The MMD is a metric on the space of probability measures.
It can be easily estimated
using the Hilbert space inner product structure.
Given two samples from the distribution of interest $x_i\sim \mu, i=1\dots M; y_j\sim \nu, j=1\dots N$,
\begin{equation}
    \begin{aligned}
            \mmd(\mu, \nu) ^2
            = \mathbb E_{x, x' \sim \mu} k(x, x')
            + \mathbb E_{y, y' \sim \nu}k(y, y') - 2\mathbb E_{x\sim \mu, y\sim \nu}k(x, y),
\end{aligned}
\end{equation}

Furthermore, the MMD has a dual formulation
 $
\mmd(\mu, \nu)=\sup_{\|f\|_\rkhs\leq 1} \int f \, d({\mu}-{\nu}),
 $
which bears the interpretation as an integral probability metric~\cite{sriperumbudur2012empirical}.

In the context of this paper, we use MMD, which is an RKHS norm, to measure the deviation of the data-driven estimation from the ground truth. For example, we later establish error analysis of the form
$
\|\hat{\mathcal{E}}{\hat{p}_{t}} - {\mathcal{E}}{{p}_{t}} \|_\mathbb{H} 
$,
where $\hat{\mathcal{E}}{\hat{p}_{t}}$ is the embedded (in $\rkhs$) data-driven estimator of the system state distribution at time $t$, and the ${\mathcal{E}}{{p}_{t}}$ is the (unknown) embedded ground-truth dynamics.
Similar to the Wasserstein metric,
the previous work \cite{zhu2021kernel} has discovered reformulation techniques for the aforementioned DRO under the MMD metric-balls
$
\left\{
\mu\ |\ \mmd(\mu, \hat \mu)\leq \rho
\right\}
$.
Compared to the DRO setting, this paper studies how to propagate such static ambiguity set through nonlinear dynamics models.
In previous studies of Koopman theory and conditional kernel embedding, researchers have focused on the empirical estimation and one-step error analysis
~\cite{parkMeasuretheoreticApproachKernel2020,
Li2022,
kosticLearningDynamicalSystems2022}
.
This paper shows that the concept of ambiguity set using the aforementioned MMD is a natural tool for establishing multistep error analysis for dynamical systems.

\section{Propagating MMD ambiguity sets through nonlinear data-driven dynamics models}

\subsection{Multistep ambiguity set propagation}
\label{sec:multi_step_bootstrap}
In the literature, error analysis of CME typically focuses on concentration properties.
For example, estimation errors of the \emph{one-step} estimators for embedded evolution operators have been studied extensively in the literature on \emph{conditional mean embeddings}, see~\cite{Fukumizu2013,parkMeasuretheoreticApproachKernel2020,Li2022,kosticLearningDynamicalSystems2022}. These results have been applied in the context of Koopman theory by~\cite{Klus2020_RKHS}.

However, in practice, error analysis of such learned models alone is not enough -- the multistep evolution of uncertainty and ambiguity must be understood.
For example, given the initial system state {$X_0$} distributed according to $p_0$, standard results, which we will expand on later, can be applied to characterize the one-step prediction error
\[
    {\|\hat{\mathcal{E}}\hat{p}_1  - {\mathcal{E}} p_{1} \|_\mathbb{H}
 = 
    \|\PhatTauHlmbd \hat{\mathcal{E}}p_0
    - \PTauHlmbd \mathcal{E} p_0\|_\mathbb{H}}
    ,
\]
where $p_1,\,\hat{p}_1$ are the densities at time $t$ obtained by applying the exact and empirical propagator, respectively.
Note that we use the notation $\hat{p}$ to indicate that this process is the estimated path of our data-driven model.
However, our interest is often to understand the prediction error after multiple steps of applying the data-driven models, i.e., characterizing 
\(
\|\hat{\mathcal{E}}\hat{p}_T  - {\mathcal{E}} p_{T } \|_\mathbb{H}
\)
for some specified integer $T \geq 1$, corresponding to propagation over physical time $T\cdot t$. Therefore, the compounding error caused by applying our error analysis repeatedly must be taken into account.
We now focus on the multistep propagation of the sampling error.
This corresponds to characterizing the error due to the randomness of empirical estimation from data.
As we shall see,
the multistep error behaves quite differently from the one-step analysis we have analyzed so far.

Our starting point is the following result.
We denote the terms
\begin{equation}
    E:= \norm{\PhatTauHlmbd}_{{L(\mathbb{H})}}\quad
    F:=\| \PhatTauHlmbd - \PTauHlmbd \|_{{L(\mathbb{H})}}.
    \label{eq-AB-op-norms}
\end{equation}
We summarize the result below.
\begin{proposition}[Propagate multistep MMD ambiguity]
    \label{thm:mmd-prop}
    Suppose $p_0, q_0$ are two initial state distributions.
    Let ${\mathcal{E}}p_T$ denote the embedded distribution propagated by applying the \emph{true unknown} push-forward operator $\PTauHlmbd$ to $p_0$ for $T$ times;
    $\hat{\mathcal{E}}\hat{q}_T$ denote the embedded distribution propagated by applying the \emph{empirical} push-forward operator $\PhatTauHlmbd$ to $q_0$ for $T$ times, i.e.,
    \begin{equation}
        {\mathcal{E}}p_T = (\PTauHlmbd) ^ T {\mathcal{E}}p_0, \quad
        \hat{\mathcal{E}}\hat{q}_T = (\PhatTauHlmbd)^T {\hat{\mathcal{E}}}q_0.
    \end{equation} 
    Then, at time $T$, the error bound in $\mmd$ is given by the formula
\begin{align}
    \mmd(\hat{\mathcal{E}}\hat{q}_T , {\mathcal{E}}p_T)
    \leq
    E^T\cdot \mmd({\hat{{\mathcal{E}}}}q_0 , {\mathcal{E}}p_0) 
    + \sum_{i=1}^{T-1}E^{i} F \EpNorm{T-i-1},
\end{align}
where the factors $E, F$ are defined in \eqref{eq-AB-op-norms}.
\end{proposition}

Note that, in practical applications, the first term reflects the level of trust in the initial empirical data samples, and can be set to the existing concentration bounds characterized in kernel methods literature such as in \cite{sriperumbudur2012empirical}.

In practice, the above formula is not yet computable since we do not know the true values of $\EpNorm{T-i-1}$.
To provide a computable error bound, we establish:
\begin{proposition}
    \label{thm:mmd-prop-computable}
    In the same setting as in Proposition~\ref{thm:mmd-prop}, we have
\begin{align}
    \label{eq-multistep-mmd-theorem-empirical-computable}
    \mmd(\hat{\mathcal{E}}\hat{q}_T , {\mathcal{E}}p_T)
    \leq
    (E+F)^T\cdot \mmd(\hat{\mathcal{E}}q_0 , {\mathcal{E}}p_0) 
    + 
    \sum_{i=1}^{T-1}(E+F)^{i} F {\|\hat{\mathcal{E}}\hat{q}_{T-i-1} \|_\mathbb{H}},
\end{align}
where the factors $E, F$ are defined in \eqref{eq-AB-op-norms}.
\end{proposition}
\begin{corollary}
    \label{thm:mmd-prop-same-dist-computable}
    The multistep empirical estimation error bound in MMD for propagating distribution $p_0$ is given by
    \begin{equation}
        \mmd(\hat{\mathcal{E}}\hat{p}_T , {\mathcal{E}}p_T)
        \leq
        \sum_{i=1}^{T-1}(E+F)^{i} F {\|\hat{\mathcal{E}}\hat{p}_{T-i-1} \|_\mathbb{H}}.
    \end{equation}
\end{corollary}
Note that all the terms on the right-hand sides in Proposition~\ref{thm:mmd-prop-computable} and Corollary~\ref{thm:mmd-prop-same-dist-computable} are computable.
Finally, we also provide an implementable recursive scheme of the above error bound in Algorithm~\ref{alg:mmd-prop-forward} for propagating the MMD error bound forward in time. We later demonstrate the multistep error bounds in numerical experiments.
It is important to note that, using Algorithm~\ref{alg:mmd-prop-forward}, the key quantities that incur computational effort in~\eqref{eq-AB-op-norms} are only estimated once at the beginning of the algorithm, resulting in computational efficiency.
In our algorithm, only one system trajectory, $\left\{\hat{p}_{i}\right\}_{i=0}^T$ is propagated through \eqref{eq:prop-nominal}, which differs from traditional bootstrap estimators for kernel regression that relies on multiple bootstrap replicate of the regression estimator, see, e.g., \cite{hastie2009elements}[Chapter~8].
The idea of the multistep operator estimation error is illustrated in Figure~\ref{fig-embedded-evo-intro}.
\begin{algorithm}[h!]
    \caption{Multistep propagation of the MMD error bound}
    \label{alg:mmd-prop-forward}
        \KwData{Initial ambiguity set radius estimate $\rho_0 \geq \mmd(\hat{\mathcal{E}}{q}_0 , {\mathcal{E}}p_0)$, initial empirical state distribution embedding $\hat{\mathcal{E}}q_0$}
    \KwResult{
        The entire ambiguity tube, i.e., embedded distributional trajectory
        $\{\hat{\mathcal{E}}\hat{q}_i\}_{i=0}^{T}$
        and MMD error bound at time $i$ 
        $$
        \rho_i := \mmd(\hat{\mathcal{E}}\hat{q}_i , {\mathcal{E}}p_i),\quad i=0, \dots, T
        $$
    }
    \textbf{Algorithm:}

    Estimate $\| \PhatTauHlmbd - \PTauHlmbd \|_{{L(\mathbb{H})}}$, either via concentration results in Section~\ref{sec:concentration} or the bootstrap scheme in Algorithm~\ref{alg:bootstrap}; compute $\PhatTauHlmbd$ and its operator norm $\|\PhatTauHlmbd\|_{{L(\mathbb{H})}}$ similarly\\
    \For{$i=0,1, \dots, T-1$}{
        Compute the next-step empirical embedding, also known as the center of the ambiguity set
            \begin{equation}
                \label{eq:prop-nominal}
                \hat{\mathcal{E}}\hat{q}_{i+1} = \PhatTauHlmbd \hat{\mathcal{E}}\hat{q}_i
            \end{equation}
        \\
        Compute the next-step error bound, also known as the ambiguity set radius using \eqref{eq:error-prop-one-step-compute}
        \[
            \rho_{i+1} = 
            \| \PhatTauHlmbd - \PTauHlmbd \|_{{L(\mathbb{H})}} \cdot (\|\hat{\mathcal{E}}\hat{q}_i\|_\rkhs+\rho_i)
            +
            \|\PhatTauHlmbd\|_{{L(\mathbb{H})}}\cdot \rho_i
        \]
    }
\end{algorithm}

\subsection{An efficient empirical bootstrap estimator of operator error}
\label{subsec:bootstrap}
To characterize the deviation of the estimated operator from the true operator in~\eqref{eq:estimation_error_pf}, researchers have typically relied on concentration inequalities, such as those provided in the previous section.
Bounds derived from concentration inequalities are well-known to be conservative, especially if applied to the composed operator $\mathcal{P}^t_{\mathbb{H}, \lambda}$, where only coarse estimates for the inversion are available.
Furthermore, error analysis in multistep ambiguity propagation is even more conservative due to the compounding of error over the time steps. Those issues continue to impede the practical use of existing error analysis, including, e.g., the results in Proposition~\ref{prop:concentration_bound}.

To remedy those issues and provide a practical approach toward error analysis, we propose a bootstrap estimator of the operator deviation in \eqref{eq:estimation_error_pf}. 
Given the training data set $(\hat{X}, \hat{Y})$, we create a bootstrap copy of the empirical estimator $\PhatTauHlmbd $, denoted as $\PtildeTauHlmbd$, by using a re-sampled (with replacement) data set $(\tilde{X}, \tilde{Y})$ (i.e., $(\tilde{X}, \tilde{Y})$ is a resampled with replacement data set of $(\hat{X}, \hat{Y})$).
We then compute the deviation $\|\PhatTauHlmbd -  \PtildeTauHlmbd\|_{{L(\mathbb{H})}}$ by the following formula
\begin{equation}
    \begin{aligned}
        \label{eq:bootstrap-formula}
        \|\PhatTauHlmbd -  \PtildeTauHlmbd\|_\LH
        =
        \sqrt{\lambda_{\max}
        \left(
            K_{ZZ}^{-\frac12}(A+B-C-C^\top)K_{ZZ}^{-\frac12}
            \right)
            }
            \\
            =
            \sqrt{\lambda_{\max}\left(K_{ZZ}^{-1}(A+B-C-C^\top)\right)}.
        \end{aligned}
\end{equation}
where we use the notation $Z:= [\hat{X}, \tilde{X}]^\top$ for the concatenated data vector, and $K_{ZZ}$ is the kernel Gram matrix for the concatenated vector $Z$,
i.e.,
$K_{ZZ} = \left[ k(z_i, z_j)\right]_{i,j=1}^m$
.
The terms $A, B,$ and $C$ are defined through
\begin{equation}
    \begin{aligned}
        &A 
        =
        K_{Z\hat X}
        \bigl(\Kxhat+{m}\lambda I)^{-1} 
        K_{\hat{Y}\hat{Y}}
        \bigl(\Kxhat+{m}\lambda I)^{-1} 
        K_{\hat X Z}
        \\
        &B=
        K_{Z\tilde X}
        \bigl(\Kxtilde+{m}\lambda I)^{-1} 
        K_{\tilde{Y}\tilde{Y}}
        \bigl(\Kxtilde+{m}\lambda I)^{-1} 
        K_{\tilde X Z}
        \\
        &C 
        =
        K_{Z \hat X }
        \bigl(\Kxhat+{m}\lambda I)^{-1} 
        K_{\hat{Y}\tilde Y}
        \bigl(\Kxtilde+{m}\lambda I)^{-1} 
        K_{\tilde X Z}
        .
    \end{aligned}
\end{equation}
The derivation is provided in Section~\ref{sec:bootstrap-proof}.

We then repeat this bootstrap process $m_b$ times.
The results can be used to determine a confidence bound for the error quantile of the operator deviation in~\eqref{eq:estimation_error_pf}, i.e., we can numerically compute the quantile $\delta$ such that
\begin{equation}
    \label{eq-bootstrap-chance-constraint}
    \mathbb{P}(\| \PhatTauHlmbd - \PTauHlmbd \|_{{L(\mathbb{H})}}\leq \delta) \ge 1-\alpha
\end{equation}
where $1-\alpha$ is a given confidence level, e.g., set to $\alpha=5\%$. This procedure is outlined in Algorithm~\ref{alg:bootstrap} and illustrated in Figure~\ref{fig-bootstrap-quantile}.
The ability to produce bootstrap estimates hinges on the advantage that the RKHS operator norm admits a straightforward computable estimate in at-most polynomial time.
This trait is not shared by some other metrics such as the Wasserstein distance.
In the existing literature, bootstrap techniques for MMD have been used to produce sharp test thresholds in two-sample tests for machine learning \cite{gretton2012kernel,NIPS2017_979d472a} as well as to produce sharp approximate ambiguity set estimation for distributionally robust optimization \cite{nemmour2022maximum}.
\begin{algorithm}[h!]
    \caption{Bootstrap estimation of operator deviation}
    \label{alg:bootstrap}
    \KwData{Training data $({X}, {Y})$, number of bootstrap samples {$m_b$}, confidence level $\alpha$}
    \KwResult{Approximate quantile for operator deviation $\delta$ such that
    \begin{equation}
        \mathbb{P}(\| \PhatTauHlmbd - \PTauHlmbd \|_{{L(\mathbb{H})}}\leq \delta) \ge 1-\alpha.
    \end{equation}}
    \For{$j = 1,\ldots , m_b$}{
      Resample a data set $(\tilde{X}, \tilde{Y})$ from $({X}, {Y})$ using sampling with replacement\;
      Compute the deviation $\|\PhatTauHlmbd -  \PtildeTauHlmbd\|_{{L(\mathbb{H})}}$ using \eqref{eq:bootstrap-formula}\;
      $\Delta[j] \gets \|\PhatTauHlmbd -  \PtildeTauHlmbd\|_{{L(\mathbb{H})}}$\;
    }
    $\Delta \gets sort(\Delta)$\;
    $\delta \gets \Delta[ceil(m_b (1-\alpha))]$\;
    \end{algorithm}

    \begin{figure}[t!]
        \centering
        \includegraphics[width=6.2cm]{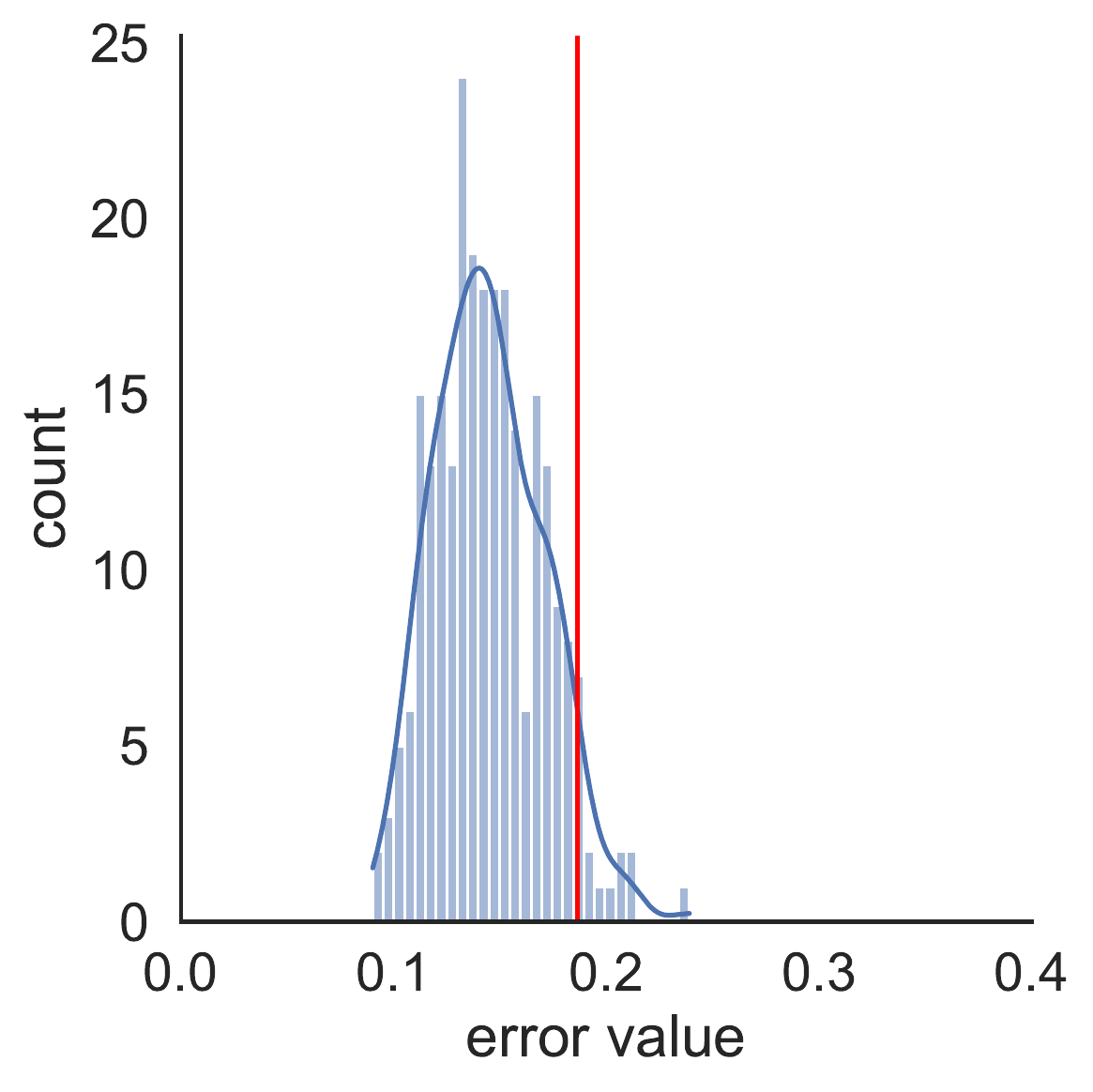}
        \caption{Bootstrapped approximate quantile for operator deviation. The horizontal axis denotes the value of the operator deviation in the bootstrapped data copy produced by the bootstrap procedures described in Section~\ref{subsec:bootstrap}. The errors are then reported in a histogram.
        The vertical redline indicates the approximate quantile bootstrap $\delta$.}
        \label{fig-bootstrap-quantile}
    \end{figure}

    \section{Numerical experiments}
    \label{sec:numerical_examples}
    \subsection{One-step prediction error analysis}
    We consider the one-dimensional Ornstein-Uhlenbeck process
    \begin{equation}
        \label{eq-ou-proc}
        dX_t = -\alpha X_t dt + \sqrt{2\beta^{-1}} dW_t,
    \end{equation}
    where $\alpha=1, \beta=1$.
    In this experiment,
    we will refer to this stochastic system as the true (unknown) dynamical system.
    
    We first time-discretize the system and simulate it forward from the sampled initial state
    $X_0^i\sim \rho_0$ until the desired time $T$.
    We then use the data
    $\{X_0^i, X_T^i\}_{i=1}^m $
    to form the embedded estimator as described in the previous sessions.
    The evolution of the system state distribution is plotted in Figure~\ref{fig-ou-process-density-shift}.
    We report the prediction error of this learned estimator in Figure~\ref{fig-ou-process-error-over-time}.
    \begin{figure}[tb]
        \centering
            \centering
            \includegraphics[height=6.5cm]{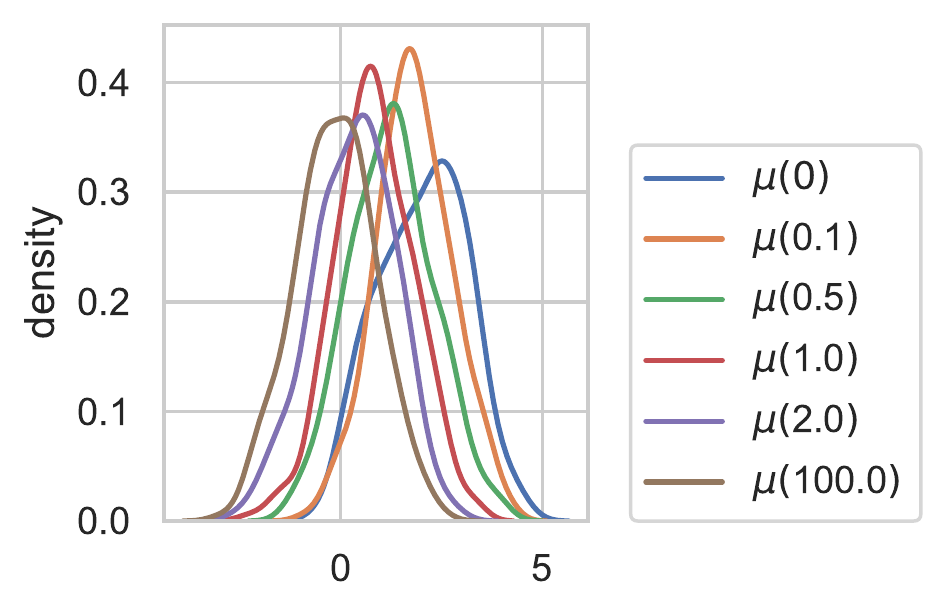}
        \caption{
            Density evolution over the lag time interval of the learned data-driven system.
            We plot the distribution from $0$ to $2$ seconds in the simulation.
            Additionally, we plot the state density at a longer time at $100$ seconds, which is close to the equilibrium state.
        }
        \vspace{-0.5cm}
        \label{fig-ou-process-density-shift}
    \end{figure}
    \begin{figure}[tb]
        \centering
        \begin{subfigure}{0.49\textwidth}
            \centering
            \includegraphics[height=6.5cm]{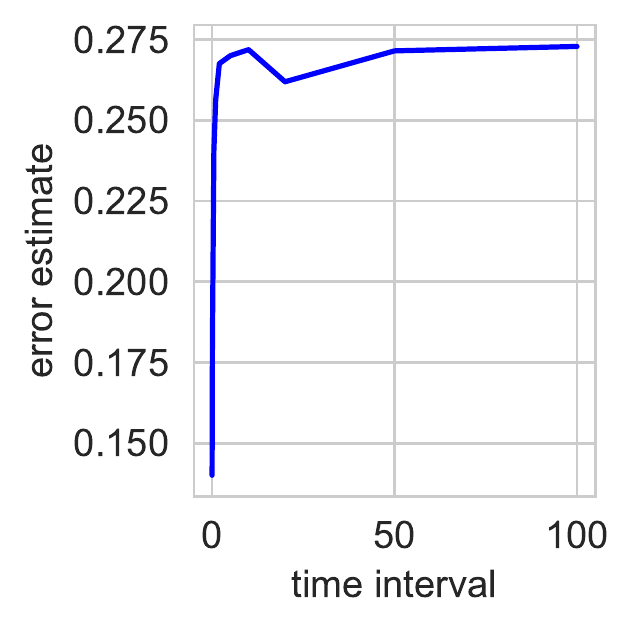}
        \end{subfigure}
        \begin{subfigure}{0.49\textwidth}
            \centering
            \includegraphics[height=6.5cm]{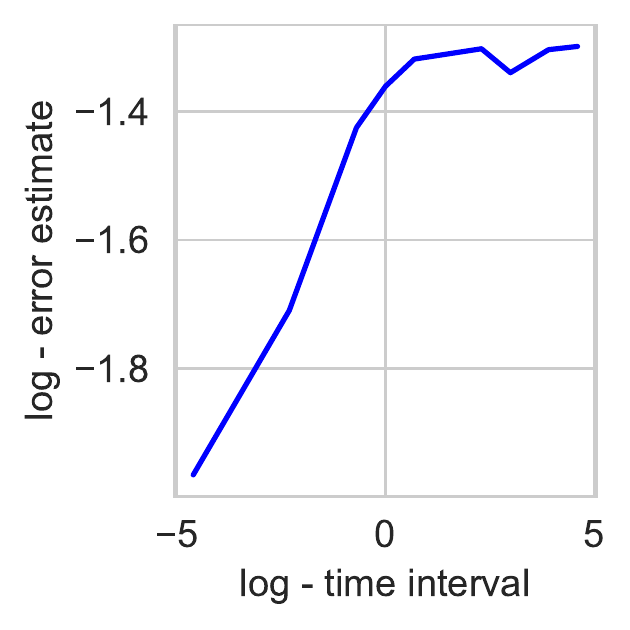}
        \end{subfigure}
            \caption{
            \textbf{(left)}
            Bootstrap error over time of the empirical estimate of the push-forward operator $\PTauHlmbd$.
            \textbf{(right)}
            We plot the same figure in a log scale to visualize the rate of change.
        }
        \vspace{-0.5cm}
        \label{fig-ou-process-error-over-time}
    \end{figure}
    
    Next, to empirically verify the concentration properties of the estimator, we plot the bootstrapped error bound as we increase the training data size.
    This is plotted in Figure~\ref{fig-ou-process-error-number-samples}.
    To further visualize the convergence rate, we plot the log-log scale in Figure~\ref{fig-ou-process-error-number-samples} bottom.
    We fit a linear regression and obtain the approximate slope $-0.4$ of the log-log curve, which is an estimate of the rate of convergence.
    This implies that the empirical bootstrap error estimate converges at the rate of approximately $n^{-0.4}$.
    \begin{figure}[tb]
        \centering
        \begin{subfigure}{0.49\textwidth}
            \centering
            \includegraphics{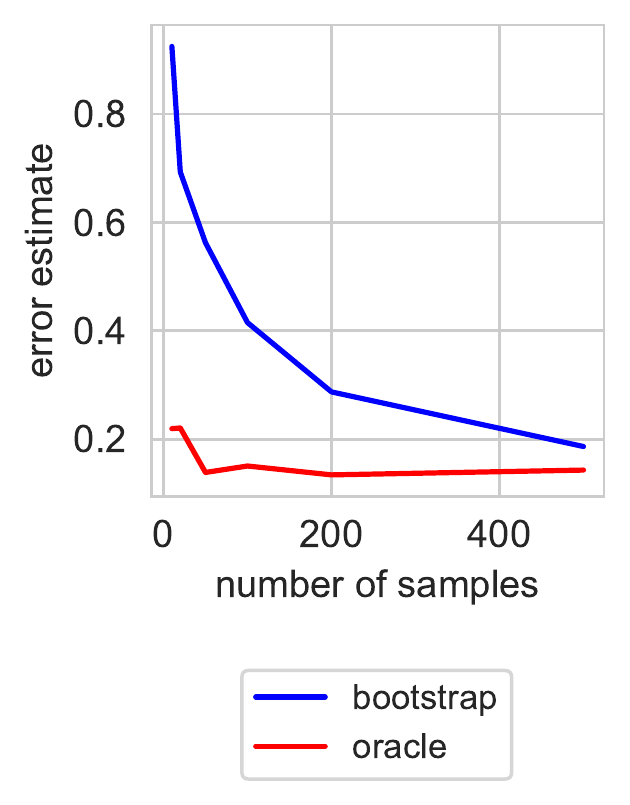}
        \end{subfigure}
            \begin{subfigure}{0.49\textwidth}
            \centering
            \includegraphics{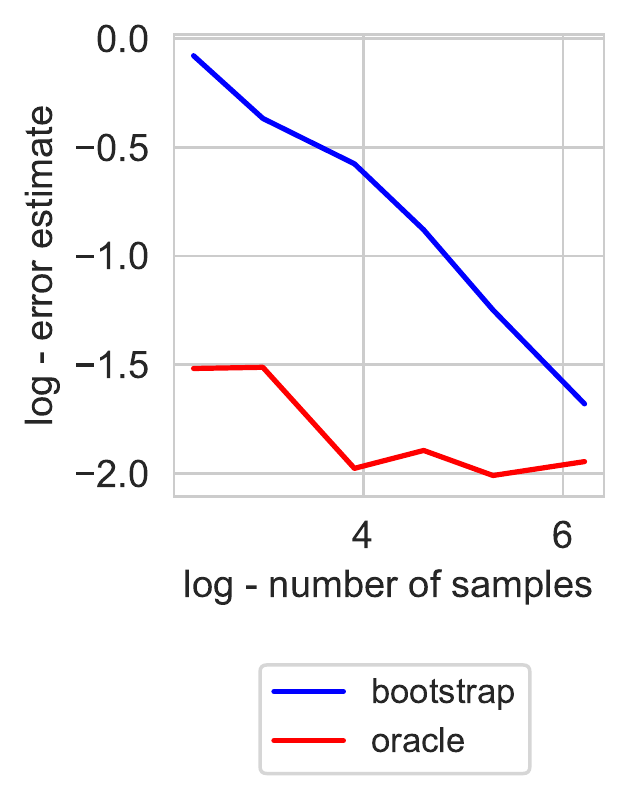}
        \end{subfigure}
        \caption{
            \textbf{(left)} Empirical bootstrap error (y-axis) plot over a varying number of training samples (x-axis).
            \textbf{(right)} Both axes are plotted in log scale to show the rate of decay.
        }
        \label{fig-ou-process-error-number-samples}
    \end{figure}
    To get a sense of how well our error estimation is, we additionally generate independent samples from the Ornstein-Uhlenbeck process~\eqref{eq-ou-proc}, denoted by
    $\{{X_0^\prime}^{i}, {X^\prime}_T^i\}_{i=1}^M $.
    We compute an approximate oracle of the deviation by the large-sample ($M=5000$) Monte Carlo estimation
    \begin{equation*}
        \mmd(\avg{M}\Phi({X^\prime}_T^i), \PhatTauHlmbd\avg{M}\Phi({X^\prime}_0^i)).
    \end{equation*}
    We compare our empirical bootstrap error with this large-sample oracle in Figure~\ref{fig-ou-process-error-number-samples}.
    We observe that the empirical bootstrap estimator does produce error bounds that tend towards the oracle as the number of training samples increases.

    \subsection{Multistep prediction error analysis}
    We now demonstrate the multistep error propagation method proposed in Section~\ref{sec:multi_step_bootstrap}.
    Figure~\ref{fig-ou-process-density-shift} depicts the evolution of RKHS embedding of the state distribution, i.e., $\hat{\mu}_{X_t}$.
    Note that those are empirical estimates of the state distribution subject to estimation error.
    \begin{figure}[tb]
        \centering
            \centering
            \begin{subfigure}{0.49\textwidth}
                \centering
                \includegraphics[width=6.2cm]{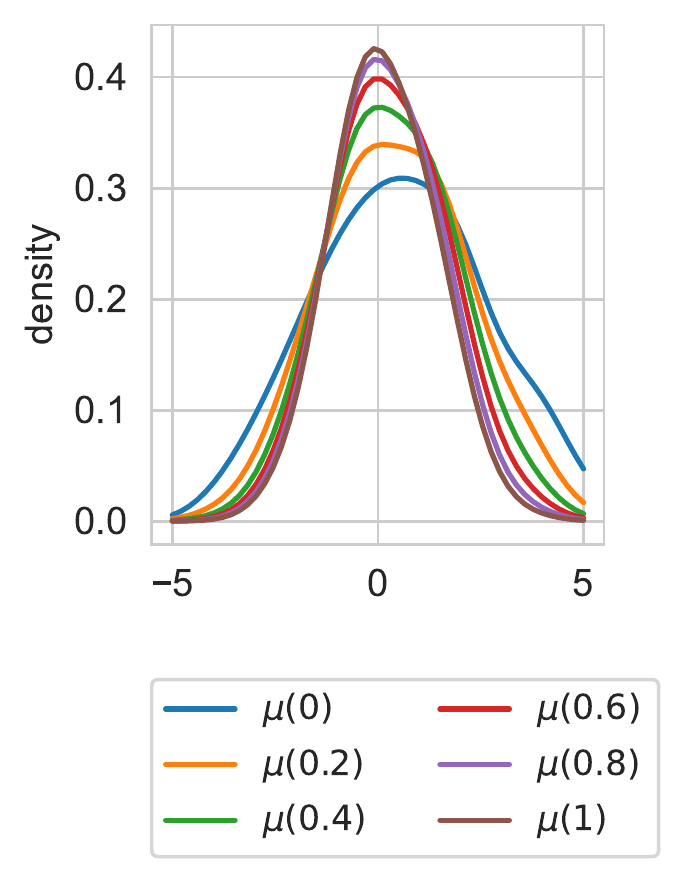}
            \end{subfigure}
            \begin{subfigure}{0.49\textwidth}
                \vspace{-0.5cm} %
                \includegraphics[width=6.2cm]{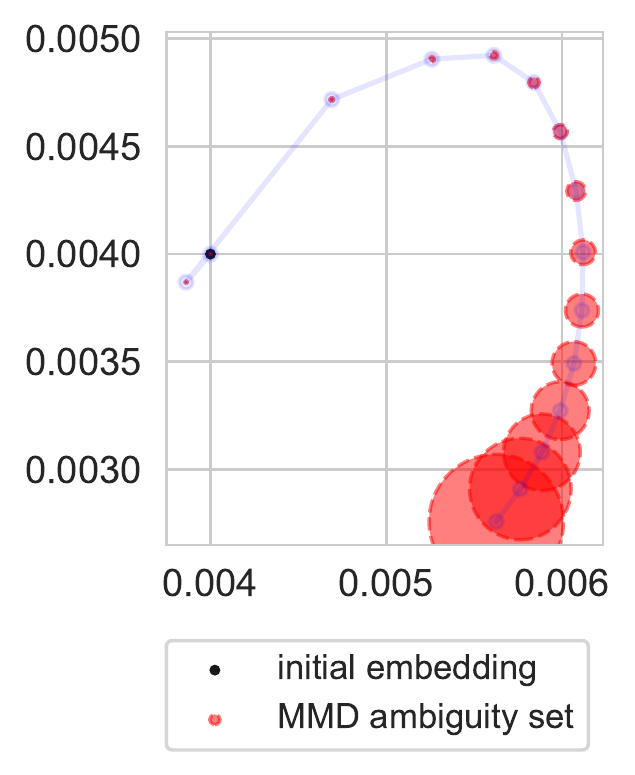}
            \end{subfigure}
        \caption{
            \textbf{(left)}
            Evolution of the distribution embedding with regression-regularization coefficient $\lambda=0.01$.
            \textbf{(right)}
            Evolution of embedded system states and error bound.
            The $x,y$-axes are the coefficients $\beta_i$ for two random components of the kernel embedding $\sum_{i=1}^{m}\beta_i\Phi({X}_t^i)$. The red balls denote the multistep MMD estimation bound, i.e., the propagated ambiguity sets.
            See also Figure~\ref{fig-embedded-evo-intro} for a similar plot, but using the span of two principal components, instead of two random components of $\beta_i$, for visualization.
        }
        \vspace{-0.5cm}
        \label{fig-density-prop}
    \end{figure}
    Different from the settings reported in \cite{Klus2020_RKHS}, which characterizes the one-step error bound,
    we compute the multistep error estimation via the ambiguity set propagation algorithm outlined in Algorithm~\ref{alg:mmd-prop-forward}.
    We set the initial ambiguity radius to be $\mmd(\mu_{X_0}, \hat{\mu}_{X_0}) = 0.1$, which can be obtained via measure concentration bounds in practice~\cite{sriperumbudur2012empirical}.
    The initial state distribution $\hat p _0$ is set to a Gaussian with
    mean $0.5$ and variance $2$.
    We then sample $m=250$ trajectories as our training data and run Algorithm~\ref{alg:mmd-prop-forward}.
    
    We now visualize our multistep error bound, i.e., the (dynamic) ambiguity sets that describe the distributional uncertainty of the embedded system states.
    Since the squared MMD distance is simply a quadratic distance in the embedded Hilbert space, we plot the \emph{ambiguity tube} (consisting of Hilbert norm-balls) centered around the empirical trajectory
    $\hat{\mathcal{E}}\hat{p}_t = \sum_{i=1}^{m}\beta_i\Phi({X}_t^i)$,
    where ${X}_t^i$ are the sampled states and $\beta_i$ are the estimated coefficients,
    computed via \eqref{eq:prop-nominal}, i.e., we plot the ambiguity sets $\mathcal{A}_t\subset \rkhs, \ {t=0}\dots{T}$, 
    \begin{align}
        \mathcal{A}_t=
        \left\{
            \mu\ |\ \|\mu - \sum_{i=1}^{m}\beta_i\Phi({X}_t^i) \|_{{\mathbb{H}}}\leq \rho_t
            \right\}
            .
    \end{align}
    Since the embedded Hilbert space $\mathbb{H}$ is infinite-dimensional, we project down to two dimensions by selecting randomly two features $\Phi({X}_t^i)$, and plot the evolution of their coefficients $\beta_i$ in Figure~\ref{fig-density-prop}~\textbf{(right)}.
    The propagated ambiguity set radii, i.e., MMD error bounds over time, are plotted in
    Figure~\ref{fig-mmd-err-prop}. See the caption therein for more details, as well as
    Figure~\ref{fig-embedded-evo-intro} for a similar illustrative plot, but using the span of two principal components, instead of two random components of $\beta_i$, for visualization.
    \begin{figure}[tb]
        \centering
        \begin{subfigure}{1\textwidth}
            \centering
            \includegraphics{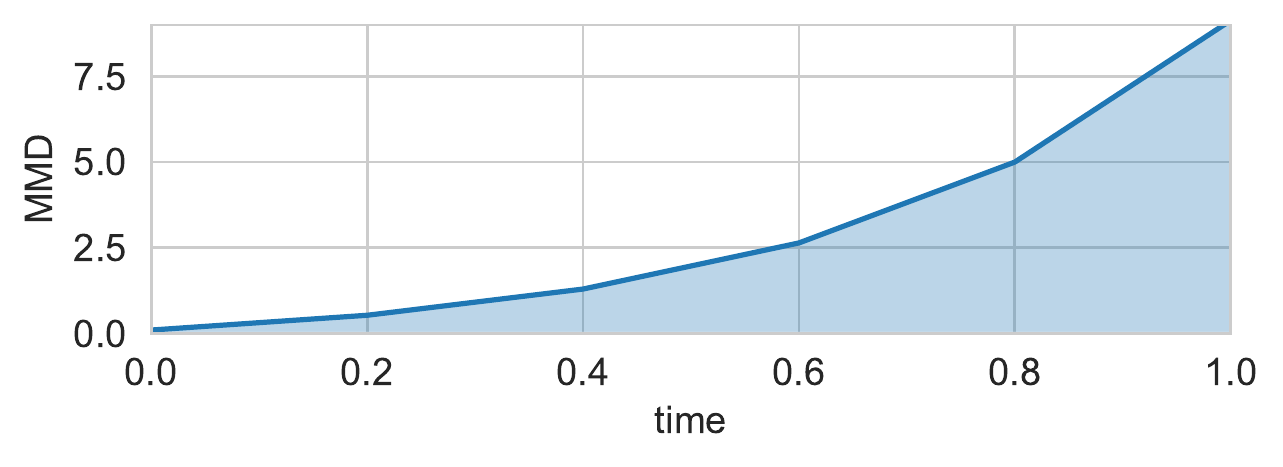}
        \end{subfigure}
        \caption{
            Multistep error estimate in MMD (ambiguity radius) plotted over time of the simulation.
            \textbf{(top)}
            The case with regression-regularization coefficient $\lambda=0.01$.
        }
        \vspace{-0.5cm}
        \label{fig-mmd-err-prop}
    \end{figure}
    
    \subsection{Further related works}
    Different from previous works that use either static ambiguity sets~\cite{mohajerin2018data,delageDistributionallyRobustOptimization2010a}
    or dynamic ambiguity sets under known piece-wise linear (or affine) dynamics~\cite{shafieezadehabadehWassersteinDistributionallyRobust2018},
    we show that our RKHS ambiguity sets are the natural modeling choice for the learned data-driven dynamical system models.
    
    The centerpiece of data-driven Koopman operator approximation is the conceptually straightforward estimation algorithm called \emph{Extended Dynamic Mode Decomposition (EDMD)}~\cite{WKR15,KKS16}, which can be used to learn a finite-dimensional approximation to evolution operators from finite simulation data. Clearly, two major sources of error arise from its application - the \emph{estimation error}, which measures the error between the computed model and the analytical Galerkin approximation onto the finite-dimensional subspace on one hand, and the \emph{approximation error}, which accounts for the data-independent projection error onto this finite-dimensional space, on the other hand. Both sources of error have been addressed previously in the literature, see~\cite{Kurdila2018,Zhang2021,Nueske2021}.
    
    Due to the close connection of the Koopman approach to machine learning, a natural extension is to represent evolution operators on a reproducing kernel Hilbert space (RKHS)~\cite{Aronszajn1950,Schoelkopf2002}, which can serve as an infinite-dimensional and therefore powerful approximation space. The first kernel-based variant of EDMD was suggested in~\cite{WRK15}, and a rigorous connection to the concept of conditional mean embedding~\cite{Fukumizu2004} was made in~\cite{Klus2020_RKHS}. In particular, it was shown that one can realize a view of the exact dynamics through the lens of an \emph{embedding operator} by a linear operator acting only on the RKHS, which can be estimated from data using kernel evaluations.
    
    The finite-data estimation error for embedded evolution operators has already been considered in the literature using concentration inequalities, see, for example,~\cite{Fukumizu2013,parkMeasuretheoreticApproachKernel2020,Li2022}. For stationary and exponentially stable systems, the analytical evolution operators often possess a low-rank structure for intermediate time windows, as all rapidly equilibrating processes have decayed at these timescales. However, intermediate time windows are often the most relevant time windows in numerous applications.

\section{Discussion and future works}
Designing closed-loop distributionally robust control for nonlinear data-driven dynamics models remains an open problem.
This paper's results lift a major roadblock in that we have proposed an algorithm that is capable of propagating MMD ambiguity sets through nonlinear data-driven dynamics models using Koopman operators and kernel CME.
This can serve as the first necessary step to apply this propagation scheme in data-driven distributionally robust control.
Furthermore, we envision that feedback distributionally robust control can further control the growth of multistep distributional ambiguity sets characterized by this paper.

\section{Acknowledgement}
This author thanks Feliks N\"uske for many helpful discussions and feedback, and Yassine Nemmour for the feedback on the manuscript.

{\small{}
\bibliographystyle{plain} %
\bibliography{Library}
}{\small\par}

\newpage
\appendix

\section{Derivation of the empirical bootstrap estimator}
\label{sec:bootstrap-proof}
Estimating the operator norm of the conditional embedding operator is straightforward and given in the following proposition and lemma.
Expanding all three terms, we obtain the bootstrap estimation in \eqref{eq:bootstrap-formula}.
Note that similar techniques have been used to estimate the kernel-constrained covariance to detect data dependence in \cite{gretton2005kernel}.
\begin{proposition}
    The norm of the difference of operator,
    $\|\PhatTauHlmbd -  \PtildeTauHlmbd\|_\LH$,
    is given by
    \begin{align*}
        \|\PhatTauHlmbd -  \PtildeTauHlmbd\|_\LH
        =
        \sqrt{\lambda_{\max}
        \left(
            K_{ZZ}^{-\frac12}(A+B-C-C^\top)K_{ZZ}^{-\frac12}
        \right)
        }
        \\
        =
        \sqrt{\lambda_{\max}\left(K_{ZZ}^{-1}(A+B-C-C^\top)\right)}.
    \end{align*}
    where we use the notation $Z:= [\hat{X}, \tilde{X}]^\top$ for the concatenated data vector, and $K_{ZZ}$ is the kernel Gram matrix for the concatenated vector $Z$.
    The constants $A, B,$ and $C$ are defined through
    \begin{equation}
        \begin{aligned}
            &A 
            =
            K_{Z\hat X}
            \bigl(\Kxhat+{m}\lambda I)^{-1} 
            K_{\hat{Y}\hat{Y}}
            \bigl(\Kxhat+{m}\lambda I)^{-1} 
            K_{\hat X Z}
            \\
            &B=
            K_{Z\tilde X}
            \bigl(\Kxtilde+{m}\lambda I)^{-1} 
            K_{\tilde{Y}\tilde{Y}}
            \bigl(\Kxtilde+{m}\lambda I)^{-1} 
            K_{\tilde X Z}
            \\
            &C 
            =
            K_{Z \hat X }
            \bigl(\Kxhat+{m}\lambda I)^{-1} 
            K_{\hat{Y}\tilde Y}
            \bigl(\Kxtilde+{m}\lambda I)^{-1} 
            K_{\tilde X Z}
            .
        \end{aligned}
    \end{equation}
\end{proposition}

\begin{proof}
    Similar to the proof of Lemma~\ref{thm:compute-op-norm},
    we write the norm estimate
    \begin{align*}
        \|
        \PhatTauHlmbd
        -
        \PtildeTauHlmbd
        \|_\LH
        &
        = 
        \sup_{\hnorm{\mu}\leq 1}
        \sqrt{
            \|
            (
                \PhatTauHlmbd
                -
                \PtildeTauHlmbd
            )
            \mu\|^2
        }
        \\
        &=
        \sup_{\hnorm{\mu}\leq 1}
        \sqrt{
            \|\PhatTauHlmbd
            \mu\|^2
            +
            \|\PtildeTauHlmbd
            \mu\|^2
            -
            2\iprod{\PhatTauHlmbd\mu}{\PtildeTauHlmbd\mu}
        } 
    \end{align*}
    Using a similar orthogonal decomposition as in the proof of Lemma~\ref{thm:compute-op-norm}, it suffices to consider the elements $\mu$ of the form
    $\mu = \sum_{i=1}^{2m} \alpha_i k(z_i, \cdot)$.
    The conclusion follows from the same derivation of the proof of Lemma~\ref{thm:compute-op-norm}.
\end{proof}

\begin{lemma}
    \label{thm:compute-op-norm}
    The CME operator norm estimation is given by
    \begin{equation}
        \begin{aligned}
            \|\PhatTauHlmbd\| = 
                \sqrt{
                    \lambda_{\max} \left(
                        K_{XX}^{\frac12}    
                    \bigl(\Kxx+m\lambda I\bigr)^{-1} K_{YY}\bigl(\Kxx+m\lambda I\bigr)^{-1}
                        K_{XX}^{\frac12}
                    \right) 
                    .
                }
        \end{aligned}
    \end{equation}
\end{lemma}
\begin{proof}
    We start by expanding the inner product
    \begin{equation}
        \begin{aligned}
            \|\PhatTauHlmbd\mu\|^2 = 
            \iprod{\PhatTauHlmbd\mu}{\PhatTauHlmbd\mu}
        \end{aligned}
    \end{equation}
    where we consider the embedding in the form of an orthogonal decomposition
    $\mu=\hat \mu + \mu^\perp$, where
    the empirical embedding is given by
        $\hat \mu=\sum_{i=1}^m \alpha_i k(x_i, \cdot)$
        and
        $\mu^\perp$ its project onto the orthogonal subspace
        with the property that
        \begin{align*}
            \PhatTauHlmbd\hat\mu
            &= \sum_{i=1}^m \alpha_i \cdot \PhatTauHlmbd k(x_i, \cdot)
            \\
            \PhatTauHlmbd\mu^\perp
            &=0
            \\
            \iprod{\hat\mu}{\mu^\perp}
                &=0.
        \end{align*}
By the definition of operator norms,
    \begin{align*}
            \|\PhatTauHlmbd\| 
            &= 
            \sup_{\hnorm{\mu}\leq 1}\|\PhatTauHlmbd\mu\| 
            \\
            &=
            \sup_{\hnorm{\hat\mu} + \hnorm{\mu^\perp}\leq 1}
            \sqrt{
                \|\PhatTauHlmbd\hat\mu\|^2
                + 2 \hiprod{\PhatTauHlmbd\hat\mu}{\PhatTauHlmbd\mu^\perp}
                + \|\PhatTauHlmbd\mu^\perp\|^2
            }
            \\
            &=
            \sup_{\hnorm{\hat\mu} + \hnorm{\mu^\perp}\leq 1}
            \sqrt{
                \|\PhatTauHlmbd\hat\mu\|^2
            }
            \\
            &= 
            \sup_{
                \alpha^\top K_{XX} \alpha\leq 1
            }
            \sqrt{
                \alpha^\top K_{XX} 
                \bigl(\Kxx+m\lambda I\bigr)^{-1} K_{YY}\bigl(\Kxx+m\lambda I\bigr)^{-1} K_{XX}
                \alpha.
            }
        \end{align*}
    This optimization problem is equivalent to a generalized eigenvalue problem, whose solution is given by
    \begin{align*}
        \sqrt{
            \lambda_{\max} \left(
                K_{XX}^{\frac12}
                \bigl(\Kxx+m\lambda I\bigr)^{-1} K_{YY}\bigl(\Kxx+m\lambda I\bigr)^{-1}
                K_{XX}^{\frac12}
            \right) 
            .
        }
    \end{align*}
    where $\lambda_{\max} $ denotes the largest singular value of a matrix.
    Hence the conclusion.
\end{proof}
If we choose those empirical points $Z$ to be the $X$ in the training data, we arrive at the empirical estimator
\begin{equation}
    \begin{aligned}
        \|\PhatTauHlmbd\| = 
            \sqrt{
                \lambda_{\max} \left(
                    K_{XX}^{\frac12}
                    \bigl(\Kxx+m\lambda I\bigr)^{-1} K_{YY}\bigl(\Kxx+m\lambda I\bigr)^{-1} K_{XX}^{\frac12}
                \right) 
                .
            }
    \end{aligned}
\end{equation}

\section{Moments of the Cross-Covariance Operator}
The moments and variance of the cross-covariance operator are central quantities for the derivation of time-dependent finite-data error bounds for embedded evolution operators. 

\subsection{Set-up}
We start by analyzing the dependence of the estimation error for embedded Koopman operators on the lag time $t$. This time dependence enters through our estimate of the cross-covariance operators. Our approach is to characterize the time-dependence of the second moment and the variance of the operator-valued random variable $c_{xy}$:
\begin{align*}
M_{2, t} &= \mathbb{E}^{\mu_{0, t}}[\|c_{xy} \|^2_{\HS}], & \sigma^2_t &= M_{2, t} - \|\mathbb{E}^{\mu_{0, t}}[c_{xy}]\|^2_{\HS} = M_{2, t} - \|\mathcal{C}_{XY}\|^2_{\HS}.
\end{align*}

To this end, we introduce a separate symbol for the diagonal of the kernel:
\begin{align}
\label{eq:two_arg_kernel}
\Phi_1(x) &= \| \Phi(x)\|_\mathbb{H}^2 = k(x, x).
\end{align}
We will require that $\Phi_1$ is square-integrable in what follows, which essentially means that $k$ must also be square-integrable, in addition to the previous assumptions on $k$. Moreover, we introduce a two-argument Koopman operator, acting on two-argument functions $\phi(x, x')$ in the product space $L^2_{\mu \otimes \mu}(\mathbb{X} \times \mathbb{X})$:
\begin{equation*}
\mathcal{K}^t_2 \phi(x, x') := \mathbb{E}^{x, x'}\left[\phi(X_t, X_t') \vert X_0 = x, X_0' = x' \right],
\end{equation*}
where the two processes $X_t, \,X_t'$ are independent. The two-argument Koopman operators form a strongly continuous semigroup on $L^2_{\mu \otimes \mu}(\mathbb{X} \times \mathbb{X})$. We also introduce the orthogonal complement of the constant function in $L^2_\mu(\mathbb{X})$ as:
\begin{equation*}
L^2_{\mu, 0}(\mathbb{X}) = \{\phi \in L^2_\mu(\mathbb{X}) \vert \, \int \phi(x)\,\mathrm{d}\mu(x) = 0 \}.
\end{equation*}
Denote its orthogonal projector by $\mathcal{P}_0$. Similarly, we denote the orthogonal projector onto the orthogonal complement of the constant function in the product space by $\mathcal{P}_{0, 2}$.
\begin{lemma}
\label{lem:characterization_hs_norms}
Let $\Phi_1 \in L^2_\mu(\mathbb{X})$. Then:\\
(i) The second moment $M_{2,t}$ is finite for all $t > 0$ and is given by 
\[ M_{2, t} = \mathbb{E}^{\mu_{0,t}}[\|c_{xy} \|^2_{\HS}] = \innerprod{\Phi_1}{\mathcal{K}^t \Phi_1}_\mu. \]
(ii) The squared Hilbert-Schmidt norm of $\mathcal{C}_{XY}$ is given by
\[  \|\mathcal{C}_{YX}^0\|^2_{\HS} = \int \int k(x, x') k(y, y') \,\mathrm{d}\mu_{0,t} (x, y)\,\mathrm{d}\mu_{0,t}(x'y') = \innerprod{k}{\mathcal{K}^t_2 k}_{\mu \otimes \mu}. \]
\end{lemma}
\begin{proof}
(i) We choose an orthonormal basis $\{e_i\}_{i=1}^\infty$ of $\mathbb{H}$ to determine the Hilbert-Schmidt norm of the rank-one operator $c_{xy}$:
\begin{align*}
\|c_{xy} \|^2_{\HS} &= \sum_{i=1}^\infty \innerprod{c_{xy} e_i}{c_{xy} e_i}_\mathbb{H} 
= \sum_{i=1}^\infty \innerprod{\Phi(y)}{e_i}_\mathbb{H}^2 \innerprod{\Phi(x)}{\Phi(x)}_\mathbb{H} \\
&=\|\Phi(x) \|^2_\mathbb{H} \|\Phi(y) \|^2_\mathbb{H} = \Phi_1(x) \Phi_1(y).
\end{align*}
using Parseval's equality in the third step. Therefore
\begin{equation*}
M_{2, t} = \mathbb{E}^{\mu_{0, t}}[\|c_{xy} \|^2_{\HS}] = \int \Phi_1(x) \Phi_1(y)\,\mathrm{d}\mu_{0,t}(x, y) = \innerprod{\Phi_1}{\mathcal{K}^t \Phi_1}_\mu,
\end{equation*}
which is finite for all $t$ if $\Phi_1 \in L^2_\mu(\mathbb{X})$.

(ii) Similarly, we determine the second term in the variance as:
\begin{align*}
\|\mathcal{C}_{XY}\|^2_{\HS} &= \sum_{i=1}^\infty \innerprod{\mathcal{C}_{XY} e_i}{\mathcal{C}_{XY} e_i}_\mathbb{H} \\
&= \sum_{i=1}^\infty \int \int \innerprod{\Phi(x)}{e_i}_\mathbb{H} \innerprod{\Phi(x')}{e_i}_\mathbb{H}
\innerprod{\Phi(y)}{\Phi(y')}_\mathbb{H}\,\,\mathrm{d}\mu_{0,t} (x, y)\,\mathrm{d}\mu_{0,t}(x'y') \\
&= \int \int k(x, x') k(y, y') \,\mathrm{d}\mu_{0,t} (x, y)\,\mathrm{d}\mu_{0,t}(x',y').
\end{align*}
\end{proof}

We can now deduce the following asymptotic behavior of the two components of the variance:

\begin{proposition}
Let $\Phi_1 \in L^2_\mu(\mathbb{X})$. Then the second moment and the norm of $\mathcal{C}_{YX}^0$ have the following asymptotics in $t$:
\begin{align*}
M_{2, t} &\leq \|\Phi_1\|_1^2 + \exp\left(-\frac{2Rt}{\beta}\right)\|\mathcal{P}_0 \Phi_1\|^2_2. \\
\|\mathcal{C}_{YX}^0 \|_{\mathrm{HS}}^2 &\leq \|k\|_1^2 + \exp\left(-\frac{2Rt}{\beta}\right)\|\mathcal{P}_{0,2} k\|^2_2.
\end{align*}
\end{proposition}
\begin{proof}
We first split the function $\Phi_1$ into its stationary and centered parts:
\begin{align*}
\Phi_1 &= \innerprod{\Phi_1}{1}_\mu \,1  + \mathcal{P}_0 \Phi_1 = \|\Phi_1\|_1 \, 1 + \mathcal{P}_0 \Phi_1.
\end{align*}
Then, by the Poincaré inequality for the second term, and the invariance of $\mathcal{K}^t$ on $L^2_{\mu, 0}(\mathbb{X})$, we get that:
\begin{align*}
|M_{2, t}| = |\innerprod{\Phi_1}{\mathcal{K}^t \Phi_1}_\mu| \leq \|\Phi_1\|_1^2 + \exp\left(-\frac{2Rt}{\beta}\right)\|\mathcal{P}_0 \Phi_1\|^2_2.
\end{align*}
For the second part, we can use that the product measure $L^2_{\mu \otimes \mu}$ also satisfies a Poincaré inequality with the same constant $R$~\cite{Lelievre2016}. Therefore, we can apply the same argument, i.e. first decompose the kernel function as
\begin{align*}
k &= \innerprod{k}{1}_\mu \,1  + \mathcal{P}_{0,2} k = \|k\|_1 \, 1 + \mathcal{P}_{0,2} k,
\end{align*}
and then apply the Poincaré inequality to the centered part:
\begin{align*}
\|\mathcal{C}_{XY}\|^2_{\HS} = |\innerprod{k}{\mathcal{K}^t_2 k}_{\mu \otimes \mu}| \leq \|k\|_1^2 + \exp\left(-\frac{2Rt}{\beta}\right)\|\mathcal{P}_{0,2} k\|^2_2.
\end{align*}
\end{proof}

\subsection{One-step prediction error estimation for Embedded Evolution Operators}
\label{sec:concentration}

Consider the embedded Perron-Frobenius operator $\mathcal{P}^t_\mathbb{H}$. The goal is to bound the regularized estimation error:
\begin{equation}
\label{eq:estimation_error_pf}
\| \hat{\mathcal{P}}^t_{\mathbb{H}, \lambda} - \mathcal{P}^t_\mathbb{H, \lambda} \|_{L(\mathbb{H})}.
\end{equation}
Applying a standard decomposition, following~\cite{Fukumizu2013}, we can write:

\begin{align}
\label{eq:decomp_estimation_err}
\| \hat{\mathcal{P}}^t_{\mathbb{H}, \lambda} - \mathcal{P}^t_\mathbb{H, \lambda} \|_{L(\mathbb{H})} &\leq 
\| (\hat{\mathcal{C}}_{YX}^0 - \mathcal{C}_{YX}^0) (\hat{\mathcal{C}}_{XX,\lambda}^0)^{-1} \|_{L(\mathbb{H})} \\
&\quad + \| \mathcal{C}_{YX}^0 (\hat{\mathcal{C}}_{XX,\lambda}^0)^{-1} (\hat{\mathcal{C}}_{XX,\lambda}^0 - \mathcal{C}_{XX,\lambda}^0) (\mathcal{C}_{XX,\lambda}^0)^{-1} \|_{L(\mathbb{H})}.
\end{align}

Using that $\|\hat{\mathcal{C}}_{XX,\lambda}^{-1}\|_{L(\mathbb{H})}$ and $\|\mathcal{C}_{XX,\lambda}^{-1}\|_{L(\mathbb{H})}$ are bounded by $\lambda^{-1}$, and that the Hilbert-Schmidt norm dominates the operator norm, we obtain:
\begin{align}
\label{eq:decomp_estimation_err_2}
\| \hat{\mathcal{P}}^t_{\mathbb{H}, \lambda} - \mathcal{P}^t_\mathbb{H, \lambda} \|_{L(\mathbb{H})} &\leq \lambda^{-1} \| \hat{\mathcal{C}}_{YX}^0 - \mathcal{C}_{YX}^0 \|_{\HS} + \lambda^{-2} \| \mathcal{C}_{YX}^0\|_{\HS} \|\hat{\mathcal{C}}_{XX}^0 - \mathcal{C}_{XX}^0 \|_{\HS}.
\end{align}

This bound can be combined with the results from above and concentration inequalities in order to obtain a probabilistic bound for the estimation error. For example, using Bernstein's inequality in Hilbert space, the following result can be shown:

\begin{proposition}
\label{prop:concentration_bound}
Let the rank-one operators $c_{xx}^0$ and $c_{yx}^0$ satisfy the following moment bounds for all $m \geq 2$ and some $L \in \mathbb{R}$:
\begin{align*}
\mathbb{E}^{\mu_{0,t}}[\|c_{yx}^0\|^m_\HS] &\leq \frac{1}{2}m! \sigma^2_t L^{m-2}, & \mathbb{E}^{\mu}[\|c_{xx}^0\|^m_\HS] \leq \frac{1}{2}m! \sigma^2_0 L^{m-2}.
\end{align*}
Then for any $\delta \geq 0$ and $m \in \mathbb{N}$, we have with probability at least $1 - 2 \delta$:
\begin{equation}
\| \hat{\mathcal{P}}^t_{\mathbb{H}, \lambda} - \mathcal{P}^t_\mathbb{H, \lambda} \|_{L(\mathbb{H})} \leq \frac{2}{\lambda \sqrt{m}}\log(\frac{2}{\delta}) \left[ \sigma_t + \frac{\|\mathcal{C}_{YX}^0\|_\HS}{\lambda} \sigma_0 + (1 + \frac{\|\mathcal{C}_{YX}^0\|_\HS}{\lambda})\frac{L}{\sqrt{m}}\right].
\end{equation}
\end{proposition}
\begin{proof}
We apply the Bernstein inequality~\cite{Caponnetto2007}[Proposition 2] separately to both terms in~\eqref{eq:decomp_estimation_err_2}, to find that with probability at least $1 - \delta$:
\begin{align*}
\| \hat{\mathcal{C}}_{YX}^0 - \mathcal{C}_{YX}^0 \|_{\HS} &\leq 2\log(\frac{2}{\delta})\left[ \frac{\sigma_t}{\sqrt{m}} + \frac{L}{m} \right], &
\| \hat{\mathcal{C}}_{XX}^0 - \mathcal{C}_{XX}^0 \|_{\HS} &\leq 2\log(\frac{2}{\delta})\left[ \frac{\sigma_0}{\sqrt{m}} + \frac{L}{m} \right].
\end{align*}
The result then follows from~\eqref{eq:decomp_estimation_err_2} and an intersection bound.
\end{proof}

\begin{remark}
The assumptions of Proposition~\ref{prop:concentration_bound} are satisfied in particular if $\Phi_1$ is uniformly bounded, by Lemma~\ref{lem:characterization_hs_norms}. In turn, this will hold if the kernel is bounded.

Note that we do not attempt to provide the sharpest possible estimate here. The estimate from Bernstein's equality nicely allows accounting for the time dependence by means of the variance. For a bounded kernel, we could use Hoeffding's inequality instead, but the time dependence would be lost in that case.
\end{remark}

\end{document}